\documentclass[english, a4paper]{article}

\RequirePackage[a4paper, heightrounded, text={14cm,22.2cm}, ignoreall]{geometry}

\usepackage{amsmath}
\usepackage{amsthm}
\usepackage{amssymb}
\usepackage{authblk}
\usepackage{hyperref}
\usepackage{microtype}
\usepackage{url} 
\usepackage{subcaption}
\usepackage[indent, skip=0.25\baselineskip plus 2pt]{parskip}

\newtheorem{observation}{Observation}
\newtheorem{conjecture}{Conjecture}
\newtheorem{lemma}{Lemma}
\newtheorem{theorem}{Theorem}
\newtheorem{proposition}{Proposition}

\usepackage{todonotes}

\title{Coloring problems on arrangements of pseudolines\footnote{
	S.~Roch was funded by the DFG-Research Training Group 'Facets of Complexity' (DFG-GRK~2434). Special thanks to Rimma Hämäläinen for the stimulating discussions on this topic.
}}

\author{Sandro Roch}
\affil{Technische Universität Berlin\\
  \texttt{roch@math.tu-berlin.de}}

\begin{document}

\maketitle

\begin{abstract}
  Arrangements of pseudolines are a widely studied generalization of line arrangements. They are defined as a finite family of infinite curves in the Euclidean plane, any two of which intersect at exactly one point. One can state various related coloring problems depending on the number $n$ of pseudolines. In this article, we show that $n$ colors are sufficient for coloring the crossings avoiding twice the same color on the boundary of any cell, or, alternatively, avoiding twice the same color along any pseudoline. We also study the problem of coloring \mbox{the pseudolines avoiding monochromatic crossings}.
\end{abstract}

\section{Introduction}

An \textit{arrangement of pseudolines} or \textit{pseudoline arrangement} is a finite family of simple continuous curves $f_1, \cdots, f_n: \mathbb{R}\to\mathbb{R}^2$ in the Euclidean plane with \[\lim_{t\to\infty}\lVert f_i(t)\rVert = \lim_{t\to-\infty}\lVert f_i(t)\rVert = \infty,\] and the property that each pair $f_i, f_j$, $i\neq j$ crosses in exactly one point. A pseudoline arrangement is \textit{simple}, if at most two pseudolines cross in a single point, see Figure~\ref{fig:example_arrangement:nonsimple} and Figure~\ref{fig:example_arrangement:simple} for examples of a non-simple and a simple arrangement of $6$ pseudolines.

\begin{figure}[tbh]
    \centering
    \begin{subfigure}[b]{0.32\textwidth}
        \centering
        \includegraphics[page=1]{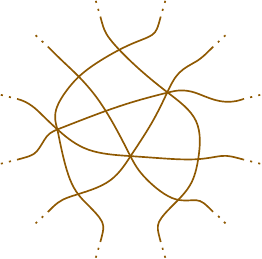}
        \caption{}
        \label{fig:example_arrangement:nonsimple}
    \end{subfigure}
    \hfill
    \begin{subfigure}[b]{0.32\textwidth}
        \centering
        \includegraphics[page=2]{figures/arrangement_and_tiling.pdf}
	\caption{}
        \label{fig:example_arrangement:simple}
    \end{subfigure}
    \hfill
    \begin{subfigure}[b]{0.32\textwidth}
        \centering
        \includegraphics[page=3]{figures/arrangement_and_tiling.pdf}
	\caption{}
        \label{fig:example_arrangement:correspondence}
    \end{subfigure}
	
    \caption{A non-simple (a) and a simple (b) pseudoline arrangement together with a corresponding tiling (c).}
    \label{fig:example_arrangement}
\end{figure}

Pseudoline arrangements are widely studied objects. They were first described in 1926 by Levi~\cite{levi26} and were further studied by Ringel~\cite{ringel57} and Grünbaum \cite{gruenbuam72}. Every line arrangement is also a pseudoline arrangement. On the other hand, there exist arrangements of at least~$n\geq 8$ pseudolines that cannot be ``strechted'', i.e. they are not isomorphic to any line arrangement, see~\cite{ringel57} and~\cite{goodmanPollack80}. But pseudoline arrangements are not only a generalization of line arrangements: Isomorphism classes of simple arrangements are in correspondence with a rich variety of other objects, such as rhombic tilings of~$2$-dimensional zonotopes (indicated in Figure~\ref{fig:example_arrangement:correspondence}), classes of reduced words of permutations and oriented matroids of rank~$3$. For a general introduction to pseudoline arrangements we refer to \cite{felsnerGoodman17}, \cite{felsnerWeil01} and \cite[ch. 6]{bjoernerEtAl99}.

\subsection{Related work}

In 2006, Felsner, Hurtado, Noy and Streinu~\cite{felsnerHurtadoNoyStreinu06} studied the \textit{arrangement graph}~$G_\mathcal{A}$ of a pseudoline arrangement~$\mathcal{A}$, which consists of the crossings in $\mathcal{A}$ as vertices and its edges are formed by the arcs between them. They give a short argument that $G_\mathcal{A}$ can be colored using three colors if $\mathcal{A}$ is simple. As $G_\mathcal{A}$ is planar, it is clearly $4$-colorable, including for non-simple arrangements. In~\cite{chiuErAl23} one can find an infinite family of line arrangements that require $4$ colors.

In 2013, Bose et al.~\cite{boseEtAl13} introduced further coloring problems on line arrangements. An arrangement decomposes the Euclidean plane into \textit{cells}: The example in Figure~\ref{fig:example_arrangement:nonsimple} consists of~$7$ \textit{bounded cells} and~$12$ \textit{unbounded cells}. One of the most remarkable results in~\cite{boseEtAl13} states that coloring the lines of a simple arrangement of $n$ lines avoiding cells whose bounding lines have all the same color requires at most $\mathcal{O}(\sqrt{n})$ colors. This was improved to $\mathcal{O}(\sqrt{n / \log n})$ by Ackerman, Pach, Pinchasi, Radoi\v{c}i\'{c} and T\'{o}th~\cite{ackermanEtAl14}, extending it also to non-simple line arrangements. Finding line arrangements that require many colors in such a coloring seems to be a difficult task; in~\cite{boseEtAl13} they provide a construction that requires $\Omega\left(\log / n\log \log n\right)$ colors.

\subsection{Results}

In \cite{boseEtAl13} and \cite{ackermanEtAl14}, the language of hypergraph coloring serves as a common formalization of the different coloring concepts and allows for the use of results from this field. If~$\mathcal{H}=(V, \mathcal{E})$ is a hypergraph, a \textit{vertex coloring} of~$\mathcal{H}$ is a coloring of the vertices avoiding \textit{monochromatic edges}, i.e. hyperedges whose contained vertices are assigned all the same color, while an \textit{edge coloring} of~$\mathcal{H}$ is a coloring of the hyperedges with no vertex being incident to two edges of the same color. The \textit{(vertex) chromatic number}~$\chi(\mathcal{H})$ is the minimal number of colors of a vertex coloring, while the \textit{edge chromatic number}~$\chi'(\mathcal{H})$ is the minimal number of colors of an edge coloring. Note that vertex coloring is not equivalent to edge coloring of the hypergraph dual.

Our results can all be stated in terms of two hypergraphs: The vertices of~$\mathcal{H}_{\text{cell-vertex}}(\mathcal{A})$ are the (bounded and unbounded) cells of~$\mathcal{A}$, and each crossing~$c$ defines a hyperedge consisting of the cells that contain $c$ on their boundary. At the same time, the hypergraph~$\mathcal{H}_{\text{line-vertex}}(\mathcal{A})$ is defined on the set of $n$ pseudolines as vertices and each crossing in~$\mathcal{A}$ defines a hyperedge consisting of the pseudolines involved in~$c$. Section~\ref{sec:coloring_crossings} is devoted to problems in which the crossings are colored. We show~$\chi'(\mathcal{H}_{\text{cell-vertex}})\leq n$ for every pseudoline arrangement:

\begin{theorem}\label{thm:crossing_coloring_face_respecting}
    Let $\mathcal{A}$ be an arrangement of $n$ pseudolines. The crossings of $\mathcal{A}$ can be colored using $n$ colors so that no color appears twice on the boundary of any cell.
\end{theorem}

\begin{figure}
    \centering
    \includegraphics{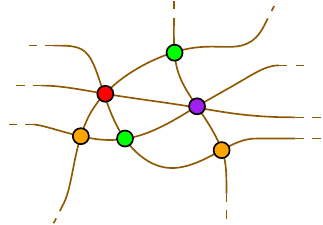}
    \caption{Coloring that fulfills the statements of Theorem~\ref{thm:crossing_coloring_face_respecting} and Theorem~\ref{thm:crossing_coloring} simultaneously.}
    \label{fig:example_crossing_coloring}
\end{figure}

 The abovementioned results in \cite{boseEtAl13} and \cite{ackermanEtAl14} are bounds on the chromatic number of a hypergraph~$\mathcal{H}_{\text{line-cell}}$ restricted to the case of line arrangements. However, none of the coloring problems that are discussed in \cite{boseEtAl13} relates lines with crossings. This is done in the following two theorems, the first one of which shows $\chi'(\mathcal{H}_{\text{line-vertex}})\leq n$:

\begin{theorem}\label{thm:crossing_coloring}
    Let $\mathcal{A}$ be an arrangement of $n$ pseudolines. The crossings of $\mathcal{A}$ can be colored using $n$ colors so that no color appears twice along any pseudoline.
\end{theorem}

Figure~\ref{fig:example_crossing_coloring} shows an example of a coloring as guaranteed in Theorem~\ref{thm:crossing_coloring_face_respecting} and in Theorem~\ref{thm:crossing_coloring}. In Section~\ref{sec:coloring_pseudolines}, we study the number of colors required to color the pseudolines avoiding monochromatic crossings. In addition to several minor results, we prove:

\begin{theorem}\label{thm:pseudoline_coloring_degree_at_least_four}
    Let $\mathcal{A}$ be an arrangement of $n$ pseudolines. The pseudolines of $\mathcal{A}$ can be colored using $\mathcal{O}(\sqrt{n})$ colors avoiding monochromatic crossings of degree at least $4$.
\end{theorem}

Here, the \textit{degree} of a crossing is the number of pseudolines that intersect in said crossing.

\section{Coloring crossings}\label{sec:coloring_crossings}

\subsection{Avoiding twice the same color on the boundary of any cell}

The following Lemma~\ref{lemma:pseudoline_face_incidence} and Lemma~\ref{lemma:acyclic} are well-known facts about pseudoline arrangements. They constitute the drawing of a pseudoline arrangement as a \textit{wiring diagram}, see~\cite{felsnerWeil01}. Assuming such a wiring diagram, our proof of Theorem~\ref{thm:crossing_coloring} will consist of a greedy coloring of the crossings from left to right. Nevertheless, we provide the simple proofs of these facts for the sake of completeness. The following Lemma~\ref{lemma:pseudoline_face_incidence} implies that a bounded (unbounded) cell contains at most $n$ ($n-1$) crossings on its boundary.

\begin{lemma}\label{lemma:pseudoline_face_incidence}
    For any cell $F$ and any pseudoline $l$ in a pseudoline arrangement, $l$ contains at most one boundary segment of $F$.
\end{lemma}
\begin{proof}
    Suppose $l$ contains two boundary segments $r$ and $s$ of $F$ as in Figure~\ref{fig:pseudoline_face_incidence}. Then there is at least one intermediate segment $t$ that belongs to a pseudoline $l'$ which must intersect $l$ at least twice. This contradicts the definition of a pseudoline arrangement.
\end{proof}

\begin{figure}[h]
    \centering
    \includegraphics{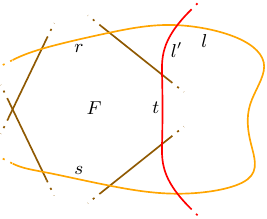}
    \caption{Proof idea of Lemma~\ref{lemma:pseudoline_face_incidence}.}
    \label{fig:pseudoline_face_incidence}
\end{figure}

In the following, we assume that a fixed unbounded cell of $\mathcal{A}$ is marked as the \textit{north cell} $N$. This induces an orientation of the pseudolines in which $N$ lies to the left of each pseudoline, see Figure~\ref{fig:proof_acyclicity:orientation}. Moreover, this induces an orientation of the \textit{arrangement graph} $G_\mathcal{A}$, whose vertices are the crossings and its edges represent the arcs between any two successive crossings, see Figure~\ref{fig:proof_acyclicity:arrangement_graph}.

\begin{figure}[tb]
    \centering
    \begin{subfigure}[b]{0.32\textwidth}
        \centering
        \includegraphics[page=1]{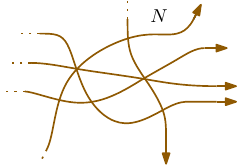}
        \caption{}
        \label{fig:proof_acyclicity:orientation}
    \end{subfigure}
    \hfill
    \begin{subfigure}[b]{0.32\textwidth}
        \centering
        \includegraphics[page=2]{figures/proof_acyclicity.pdf}
	\caption{}
        \label{fig:proof_acyclicity:arrangement_graph}
    \end{subfigure}
    \hfill
    \begin{subfigure}[b]{0.32\textwidth}
        \centering
        \includegraphics[page=3]{figures/proof_acyclicity.pdf}
	\caption{}
        \label{fig:proof_acyclicity:ccw_cycle}
    \end{subfigure}
	
    \caption{The orientation induced by fixing a north cell is acyclic.}
    \label{fig:proof_acyclicity}
\end{figure}

\begin{lemma}\label{lemma:acyclic}
    The directed graph $G_\mathcal{A}$ is acyclic.
\end{lemma}
\begin{proof}
    Number the pseudolines in counterclockwise order, starting from~$N$, as shown in Figure~\ref{fig:proof_acyclicity:arrangement_graph}. Suppose there is a counterclockwise oriented cycle of length~$k$ consisting of arcs that belong to the pseudolines~$l_1, \cdots, l_k, l_{k+1} = l_1$ in this order. For any~$i$, $l_{i+1}$ crosses~$l_i$ from right to left, which implies~$l_i \leq l_{i+1}$, see Figure~\ref{fig:proof_acyclicity:ccw_cycle} for an illustration. Hence~$l_1 \leq \cdots \leq l_k \leq l_{k+1} = l_1$, therefore~$l_1 = \cdots = l_k$, which is a contradiction. The argument for a clockwise cycle is analogous.  
\end{proof}

Fix a topological sorting $\pi$ of $G_\mathcal{A}$, which exists because of the fact that $G_\mathcal{A}$ is acyclic (Lemma~\ref{lemma:acyclic}). We aim for coloring the crossings of $\mathcal{A}$ greedily in the order of $\pi$. For any crossing $c$, a \textit{conflict ancestor} is a crossing $c'$ that comes before $c$ w.r.t. $\pi$ and both $c$ and $c'$ are on the boundary of a common cell. Figure~\ref{fig:conflict_ancestors:definition} shows an example: The red crossings are conflict ancestors of $c$; the orange crossing may be a conflict ancestor depending on $\pi$.

\begin{figure}[tb]
    \centering
    \begin{subfigure}[b]{0.38\textwidth}
        \centering
        \includegraphics[page=1]{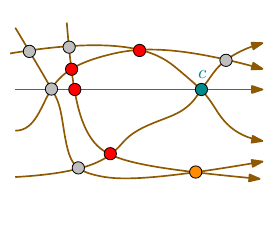}
        \caption{}
        \label{fig:conflict_ancestors:definition}
    \end{subfigure}
    \hfill
    \begin{subfigure}[b]{0.28\textwidth}
        \centering
        \includegraphics[page=2]{figures/conflict_ancestors_appendix.pdf}
	\caption{}
        \label{fig:conflict_ancestors:case_unbounded_cell}
    \end{subfigure}
    \hfill
    \begin{subfigure}[b]{0.32\textwidth}
        \centering
        \includegraphics[page=3]{figures/conflict_ancestors_appendix.pdf}
	\caption{}
        \label{fig:conflict_ancestors:case_bounded_cell}
    \end{subfigure}
	
    \caption{(a): Example for \textit{conflict ancestors}; (b), (c): Distinction between the cases of $F$ being an unbounded or a bounded cell in the proof of Lemma~\ref{lemma:bound_conflict_ancestors}.}
    \label{fig:conflict_ancestors}
\end{figure}

\begin{lemma}\label{lemma:bound_conflict_ancestors}
    Independently of the choice of the topological sorting, every crossing has at most~$n-1$ conflict ancestors.
\end{lemma}
\begin{proof}
    Let $c$ be a crossing of degree~$k$ in~$\mathcal{A}$, and let~$p_1,\cdots, p_k$ denote the involved pseudolines. Let~$\mathcal{A'} := \mathcal{A} - \{p_1, \cdots, p_k\}$ be the arrangement obtained by dropping~$p_1, \cdots, p_k$. Let~$F$ be the cell of~$ \mathcal{A'}$ in whose area~$c$ lies in~$\mathcal{A}$. Every conflict ancestor~$c'$ of~$c$ is a crossing on the boundary of~$F$, which is either also a crossing in~$\mathcal{A}'$ (type~I), or arises as a crossing when reinserting~$p_1, \cdots, p_k$ (type~II). We distinguish two cases, which are both illustrated in Figure~\ref{fig:conflict_ancestors:case_unbounded_cell} and Figure~\ref{fig:conflict_ancestors:case_bounded_cell}. In both Figures, crossings that can be conflict ancestors of~$c$ are colored yellow (type~I) or green (type~II).

    \begin{itemize}
        \item Case 1: The cell~$F$ is unbounded. By Lemma~\ref{lemma:pseudoline_face_incidence},~$F$ contains at most~$n-k-1$ crossings on its boundary, which bounds the conflict ancestors of type~I. Reinserting~$p_1, \cdots, p_k$ in~$\mathcal{A}'$ adds at most~$2k$ crossings on the boundary of~$F$, of which at most~$k$ are incoming neighbors of~$c$. The outgoing ones cannot be conflict ancestors of~$c$, as they appear after~$c$ in every topological sorting. Hence, the conflict ancestors of type~II are bounded by~$k$. In total,~$c$ has at most~$n-1$ conflict ancestors.

        \item Case 2: The cell~$F$ is bounded. By Lemma~\ref{lemma:pseudoline_face_incidence},~$F$ contains at most~$n-k$ crossings on its boundary, which bounds the conflict ancestors of type~I. As in the previous case, conflict ancestors of type~II are bounded by $k$. So far this argument shows that~$c$ has at most~$n$ conflict ancestors in total. We have to argue that at least one of the crossings in~$\mathcal{A'}$ that lie on the boundary of~$F$ cannot be a conflict ancestor of type~I. Observe that in~$G_\mathcal{A}$ there always exists a directed path from~$c$ to some crossing~$t$ in~$\mathcal{A'}$ which lies on the boundary of $F$: Starting from~$c$, take any outgoing arc to some crossing~$c'$. If~$c'$ is not already a crossing in~$\mathcal{A'}$, then follow the pseudoline that forms in~$\mathcal{A}'$ the boundary segment of~$F$ on which $c'$ lies until reaching a crossing in~$\mathcal{A'}$ (See path~$c$-$c'$-$t$ in the example in Figure~\ref{fig:conflict_ancestors:case_bounded_cell}). In every topological sorting,~$t$ comes after~$c$, so~$t$ cannot be a conflict ancestor of~$c$.\qedhere
    \end{itemize}
\end{proof}

\begin{proof}[Proof of Theorem~\ref{thm:crossing_coloring_face_respecting}]
    Mark an arbitrary unbounded cell in $\mathcal{A}$ as the north cell. Choose an arbitrary topological sorting $\pi$ of the corresponding acyclic orientation of $G_\mathcal{A}$ (Lemma~\ref{lemma:acyclic}). Color the crossings in this order: When coloring a crossing $c$, we cannot use a color that was already assigned to a crossing $c'$ which lies together with $c$ on the boundary of a common cell. These are exactly the conflict ancestors, whose number is bounded by~$n-1$ in Lemma~\ref{lemma:bound_conflict_ancestors}. So among $n$  colors, there is always a spare color available.
\end{proof}

The bound in Theorem~\ref{thm:crossing_coloring_face_respecting} is tight: Consider a convex $n$-gon~$F$ without parallel sides. Extending the sites in a straight way to infinity defines a line arrangement in which~$F$ is a cell that has $n$ crossings on its boundary, see Figure~\ref{fig:pentagon}. Hence, avoiding twice the same color on the boundary of~$F$ requires~$n$ colors.

\begin{figure}
    \centering
    \includegraphics{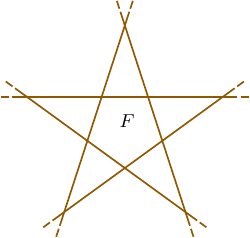}
    \caption{Construction that shows that Theorem~\ref{thm:crossing_coloring_face_respecting} is tight.}
    \label{fig:pentagon}
\end{figure}

\begin{figure}[h]
    \centering
    \includegraphics{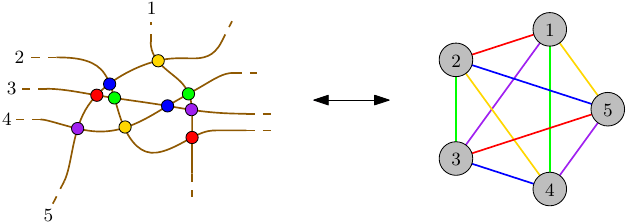}
    \caption{Coloring simple arrangements is equivalent to edge-coloring of $K_n$.}
    \label{fig:coloring_simple_arrangement}
\end{figure}

\subsection{Avoiding twice the same color along any pseudoline}

In view of Theorem~\ref{thm:crossing_coloring}, we now focus on coloring the crossings of a pseudoline arrangement~$\mathcal{A}$ avoiding twice the same color along any pseudoline. If~$\mathcal{A}$ is simple, this is equivalent to edge-coloring of the complete graph $K_n$, see Figure~\ref{fig:coloring_simple_arrangement}. It is a well-known fact that if $n$ is odd one requires exactly $n$ colors, and if $n$ is even then $n-1$ colors are sufficient. So Theorem~\ref{thm:crossing_coloring} is interesting in particular for non-simple arrangements. Unfortunately, we were unable to find an elementary proof or a simple deterministic algorithm for establishing such a coloring in the non-simple case, even though intuitively it seems to be the easier case, as less crossings appear on each pseudoline.

In 1972, Erd{\H{o}}s, Faber and Lovász met on a party and came up with the following problem (see~\cite{erdoes81}): Let $A_1, \cdots, A_n$ be a family of sets, each of cardinality $n$, and $\lvert A_i\cap A_j\rvert \leq 1$ for all $i < j$. Can one color the elements of $\bigcup_i A_i$ using $n$ colors such that each $A_i$ contains all colors? This became known as the Erd{\H{o}}s-Faber-Lovász conjecture. It was eventually proven in 2021 by Kang et al. \cite{kangEtAl23}. In hypergraph language, it is equivalent to the following statement:

\begin{theorem}[D. Y. Kang, T. Kelly, D. Kühn, A. Methuku \& D. Osthus, 2021]\label{thm:erdosFaberLovaszConjecture}
    \hspace{0.01cm}\\
    For every simple hypergraph $\mathcal{H}$ with $n$ vertices, $\chi'(H)\leq n$.
\end{theorem}

Here, a hypergraph $\mathcal{H}=(V, \mathcal{E})$ is \textit{simple} if all hyperedges have cardinality at least~$2$ and for all~$E_1, E_2\in\mathcal{E}$,~$E_1 \neq E_2$ it holds that~$\lvert E_1\cap E_2\rvert\leq 1$.

\begin{proof}[Proof of~Theorem~\ref{thm:crossing_coloring}]
    The statement is equivalent to the existence of an edge-coloring of the hypergraph~$\mathcal{H}_{\text{line-vertex}}$ using $n$ colors. $\mathcal{H}_{\text{line-vertex}}$ is simple, because there can be at most one pseudoline passing through any pair of crossings, otherwise this would mean a pair of pseudolines crossing twice. Then the statement follows from Theorem~\ref{thm:erdosFaberLovaszConjecture}.
\end{proof}

It would be nice to have a bound on the required number of colors that also takes into account how far the arrangement is away from being a simple arrangement. For this purpose, we introduce $\operatorname{mx}(\mathcal{A})$, which is defined as the maximal number of crossings along any pseudoline in~$\mathcal{A}$. For simple arrangements we have $\operatorname{mx}(\mathcal{A})=n-1$. If one excludes \textit{trivial arrangements} in which all pseudolines cross in a single point ($\operatorname{mx}(\mathcal{A})=1$), then the parameter $\operatorname{mx}$ bounds the number of pseudolines from above. This is why~$\operatorname{mx}(\mathcal{A})$ can be interpreted as an alternative measure of the size of an arrangement. In fact it was shown recently in~\cite{dumitrescu2023} that the number of pseudolines is linearly bounded by $\operatorname{mx}$, in particular $n \leq 845\cdot \operatorname{mx}(\mathcal{A})$ for large values of $n$. A weaker quadratic upper bound can be obtained as a simple observation:

\begin{observation}\label{obs:bound_on_mx}
    For every non-trivial arrangement $\mathcal{A}$ of $n$ pseudolines we have \[n \leq \operatorname{mx}(\mathcal{A})\left(\operatorname{mx}(\mathcal{A}) - 1\right)+1 \leq \left(\operatorname{mx}(\mathcal{A})\right)^2.\]
\end{observation}
\begin{proof}
    Let~$l$ be an arbitrary pseudoline of~$\mathcal{A}$. Every other pseudoline crosses~$l$ in one of at most $\operatorname{mx}(\mathcal{A})$ crossings. Let $c$ be any such crossings. The pseudolines that cross~$l$ in~$c$ must cross any other pseudoline~$l'\neq l$ in distinct crossings; see Figure~\ref{fig:mx:bound} for an illustration. Hence, their number is bounded by~$\operatorname{mx}(\mathcal{A})-1$, as~$l'$ has one further crossing with~$l$ and at most~$\operatorname{mx}(\mathcal{A})$ crossings in total. Therefore, the number of pseudolines different to~$l$ is bounded by $\operatorname{mx}(\mathcal{A})\left(\operatorname{mx}(\mathcal{A}) - 1\right)$.
\end{proof}

\begin{figure}[tb]
    \centering
    \begin{subfigure}[b]{0.55\textwidth}
        \centering
        \includegraphics{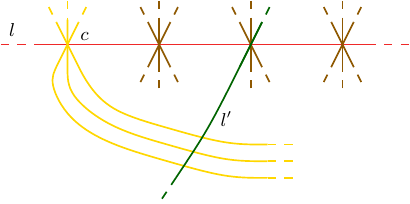}
        \caption{}
        \label{fig:mx:bound}
    \end{subfigure}
    \hfill
    \begin{subfigure}[b]{0.4\textwidth}
        \centering
        \includegraphics{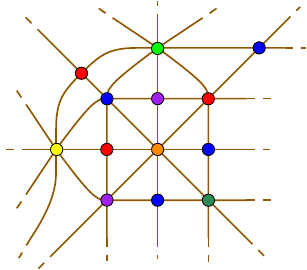}
	\caption{}
        \label{fig:mx:example_gap_3}
    \end{subfigure}
	
    \caption{(a): Illustration for the proof of Observation~\ref{obs:bound_on_mx}; (b): Arrangement with $\operatorname{mx} = 4$ that requires $7$ colors.}
    \label{fig:mx}
\end{figure}

\begin{conjecture}\label{conj:crossing_coloring_dep_on_mx}
    There is a constant $c$ so that the crossings of every pseudoline arrangement~$\mathcal{A}$ can be colored using $\operatorname{mx}(\mathcal{A})+c$ colors so that no color appears twice along any pseudoline.
\end{conjecture}

Figure~\ref{fig:mx:example_gap_3} shows a coloring of an arrangement~$\mathcal{A}$ with~$9$ pseudolines and~$\operatorname{mx}(\mathcal{A})=4$ using~$7$ colors. One can verify that this coloring is minimal. Hence, a constant as in Conjecture~\ref{conj:crossing_coloring_dep_on_mx} had to fulfill~$c\geq 3$. We were unable to find any arrangement where this gap is larger. If such a constant exists we expect it to be small.

We make use of the following result by Kahn \cite{kahn96, kangKellyKuehnMethukuOsthus21} in order to show that Conjecture~\ref{conj:crossing_coloring_dep_on_mx} holds at least asymptotically and under a certain restriction. A hypergraph $\mathcal{H}$ is \textit{$k$-bounded} if all hyperedges have cardinality at most $k$. Moreover, the \textit{codegree} $\operatorname{codeg}(\mathcal{H})$ is defined as the maximum intersection $\lvert E_1 \cap E_2\rvert$ over all pairs of hyperedges $E_1, E_2\in\mathcal{E}$, $E_1\neq E_2$.

\begin{theorem}[Kahn, 1996]\label{thm:kahn_approximate_edge_coloring}
    For every~$k, \varepsilon > 0$, there exists~$\delta > 0$ such that the following holds: If~$\mathcal{H}$ is a~$k$-bounded hypergraph of maximum degree at most~$D$ and~$\operatorname{codeg}(\mathcal{H})\leq\delta\cdot D$, then~$\mathcal{H}$ can be list edge colored using~$(1+\varepsilon)\cdot D$ colors.
\end{theorem}

The following Proposition~\ref{prop:crossing_coloring_approximately} is a simple application of Theorem~\ref{thm:kahn_approximate_edge_coloring}.

\begin{proposition}\label{prop:crossing_coloring_approximately}
    For every~$k, \varepsilon > 0$, there exists an~$\text{mx}_0\in\mathbb{N}$ so that the following\linebreak[4] holds: If a pseudoline arrangement~$\mathcal{A}$ only contains crossings of degree at most~$k$ and fulfills~\mbox{$\operatorname{mx}(\mathcal{A})\geq \text{mx}_0$}, then its crossings can be colored using~$(1+\varepsilon)\cdot \operatorname{mx}(\mathcal{A})$ colors so that no color appears twice along any pseudoline.
\end{proposition}
\begin{proof}
    Let $k, \varepsilon > 0$. Applying Theorem~\ref{thm:kahn_approximate_edge_coloring} provides a certain $\delta > 0$. Set $\text{mx}_0 := 1 / \delta$. 

    Now let $\mathcal{A}$ be a pseudoline arrangement that only contains crossings of degree at most $k$ and with $\operatorname{mx}(\mathcal{A}) \geq \text{mx}_0$. The maximum degree $D$ of~$\mathcal{H}_{\text{line-vertex}}$ equals $\operatorname{mx}(\mathcal{A})$. Moreover, the codegree of~$\mathcal{H}_{\text{line-vertex}}$ is exactly $1$, because two crossings in $\mathcal{A}$ can have at most one pseudoline in common, hence~$\operatorname{codeg}(\mathcal{H}_{\text{line-vertex}})=1\leq \delta \cdot D$. Finally,~$\mathcal{H}_{\text{line-vertex}}$ is $k$-bounded, as each crossing in~$\mathcal{A}$ is of degree at most $k$. Then, by the statement of Theorem~\ref{thm:kahn_approximate_edge_coloring}, the hypergraph~$\mathcal{H}_{\text{line-vertex}}$ can be list edge colored using~$(1+\varepsilon)\cdot \operatorname{mx}(\mathcal{A})$ colors. By the simple fact that list edge coloring is stronger than edge coloring, this implies the existence of the desired coloring of $\mathcal{A}$ using the same number of colors.
\end{proof}

\section{Coloring pseudolines}\label{sec:coloring_pseudolines}

A \emph{pseudoline coloring} of an arrangement~$\mathcal{A}$ is defined as a coloring of the pseudolines in~$\mathcal{A}$ such that there are no monochromatic crossings, i.e. crossings of pseudolines of a single color class. We let~$\chi_{pl}(\mathcal{A})$ denote the minimal number of colors in a pseudoline coloring of~$\mathcal{A}$.

\subsection{Pseudoline colorings and ordinary points}

The study of pseudoline colorings is closely related to the study of \emph{ordinary points}. An ordinary point is defined as a crossing of exactly two pseudolines, also known as \textit{simple crossing}. Every non-trivial pseudoline arrangement contains at least $\lceil 6n/13\rceil$ ordinary points~\cite{lechner88}. Two pseudolines that cross each other in an ordinary point must be assigned different colors. Hence, for simple arrangements $\mathcal{A}$ we have $\chi_{pl}(\mathcal{A})=n$.

In the following we want to take a closer look at the relationship between~$\chi_{pl}(\mathcal{A})$ and the structure of the ordinary points in a pseudoline arrangement. For this purpose, we define the \textit{ordinary graph}~$G_o(\mathcal{A})$ that has the~$n$ pseudolines of~$\mathcal{A}$ as its vertices and two of them share an edge if and only if they cross each other in an ordinary point. Clearly,~$\chi_{pl}(\mathcal{A}) \geq \chi(G_o(\mathcal{A}))$.

As a first natural question we ask how many ordinary points a pseudoline arrangement can have if it admits a pseudoline coloring with $k$ colors. Let~$\sigma_k(n)$ denote the maximum number of ordinary points of an arrangement~$\mathcal{A}$ of $n$ pseudolines with $\chi_{pl}(\mathcal{A})\leq k$.

\begin{proposition}
We have $\sigma_k(n)\in \Theta(n^2)$. More precisely, let~$t_{k}(n)$ denote the Turán number, i.e. the maximum number of edges that a graph on~$n$ vertices without containing a~$(k+1)$-clique can have. Then we have $t_{k}(n)-n\leq \sigma_{k}(n) \leq t_{k}(n)$.
\end{proposition}
\begin{proof}
    For the upper bound, suppose that an arrangement of pseudolines~$\mathcal{A}$ has more than~$t_{k}(n)$ ordinary points, so~$G_o(\mathcal{A})$ has more than~$t_{k}(n)$ edges. By Turán's theorem,~$G_o(\mathcal{A})$ contains a~$(k+1)$-clique, so the corresponding pseudolines~$p_1, \cdots, p_{k+1}$ pairwise cross in ordinary points, and hence, must be colored using at least~$k+1$ different colours.

    For the lower bound, take any simple arrangement of~$k$ pseudolines and replace each of them by a strip of~$\lfloor n/k \rfloor$ or~$\lceil n/k \rceil$ parallel pseudolines, so that one obtains an ``arrangement'' of~$n$ curves in total. So far, the ordinary graph of this ``arrangement'' is identical to the Turán graph having~$t_k(n)$ edges. In order to make it a proper pseudoline arrangement, one has to twist each of the strips, as illustrated in Figure~\ref{fig:twisted_bundles} on an example with~$k=4$ and~$n=14$. This reduces the number of ordinary points by~$n$. The so obtained arrangement~$\mathcal{A}$ clearly fulfills~$\chi_{pl}(\mathcal{A})\leq k$.
\end{proof}

\begin{figure}
    \centering
    \includegraphics{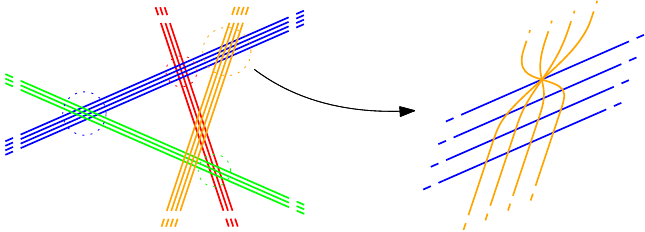}
    \caption{Constructing an arrangement with $t_k(n)-n$ ordinary points.}
    \label{fig:twisted_bundles}
\end{figure}

We would like to know how much $\chi_{pl}(\mathcal{A})$ and $\chi(G_o(\mathcal{A}))$ can differ and make the folloring observation:

\begin{proposition}\label{prop:gap}
    There are arbitrarily large arrangements with $\chi_{pl}(\mathcal{A})=2\cdot\chi(G_o(\mathcal{A}))$.
\end{proposition}
\begin{proof}
    Start with a simple arrangement of~$r$ pseudolines. Replace each pseudoline by a strip of~$3$ pseudolines, each of them being twisted in an additional crossing of degree~$3$. The ordinary graph of the so obtained arrangement~$\mathcal{A}$ of~$3r$ pseudolines is the complete~$r$-partite graph~$K_{3,\cdots, 3}$, which can be~$r$-colored. However, for a pseudoline coloring, one requires~$2$ distinct colors for each of the~$r$ strips, so~$\chi_{pl}(\mathcal{A}) = 2r$. Figure~\ref{fig:three_bundes} gives an example.
\end{proof}

\begin{figure}
    \centering
    \includegraphics{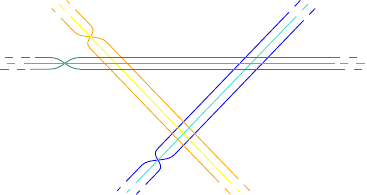}
    \caption{Arrangement $\mathcal{A}$ with $\chi_{pl}(\mathcal{A}) = 6$ and $\chi(G_o(\mathcal{A}))=3$.}
    \label{fig:three_bundes}
\end{figure}

It is unknown to us whether the factor of $2$ in Proposition~\ref{prop:gap} can be further improved.

\vspace{0.1cm}Next, we want to analyse the computatiponal complexity of~$\chi_{pl}(\mathcal{A})$. The following lemma serves as an NP-hardness reduction.

\begin{lemma}\label{lemma:reduction}
    For every graph $G$ on $n$ vertices and $m$ edges, there exists an arrangement~$\mathcal{A}_G$ of $n + \binom{n}{2} - m+1$ pseudolines with $\chi_{pl}(\mathcal{A}_G) \in \{\chi(G) + 1, \chi(G) + 2\}$ that can be computed efficiently.
\end{lemma}

\begin{proof}
    Let $G=(V, E)$ be a graph. Assume $V=[n]$ and let $m=\lvert E\rvert$. Consider an arbitrary simple arrangement~$\mathcal{A}$ of~$n$ pseudolines in which the pseudolines are drawn as~$x$-monotone curves and no two crossings lie on a vertical line, see for example the brown pseudolines in Figure~\ref{fig:np_reduction}. Label these pseudolines by~$l_1, \cdots, l_n$.  For all~$1 \leq i<j\leq n$ with~$\{ i, j \}\notin E$ insert a pseudoline~$L_{i,j}$ (green) by drawing a vertical line going through the crossing between $l_i$ and $l_j$ and let all these new pseudolines $L_{i,j}$ cross each other in a common point $p$ lying above~$l_1, \cdots, l_n$. Insert one further pseudoline~$L^*$ (orange) which crosses~$l_1, \cdots, l_n$ in ordinary points and the pseudolines~$L_{i,j}$ in $p$. 

    It remains to show that the so obtained arrangement $\mathcal{A}_G$ on $n + \binom{n}{2} - m + 1$ pseudolines fulfills \mbox{$\chi(G) + 1 \leq \chi_{pl}(\mathcal{A}) \leq \chi(G) + 2$}. For the upper bound, given a proper $k$-coloring~$c$ of~$G$, observe that the following coloring is a valid pseudoline coloring of~$\mathcal{A}_G$: Assign color~$c(i)$ to~$l_i$, to each pseudoline~$L_{i,j}$ assign an extra color~$k+1$ and to pseudoline~$L^*$ a further extra color~$k+2$. For the lower bound, assume there exists a coloring~$c$ of~$\mathcal{A}_G$ using only~$\chi(G)$ colors. For each~$(i, j)\in E$, pseudolines~$l_i$ and~$l_j$ must have different colors, since they cross in an ordinary point. Therefore, by the minimality of~$\chi(G)$, all~$\chi(G)$ colors appear on~$l_1, \cdots, l_n$ at least once. As~$L^*$ crosses all of them in ordinary points, a further color is required, contradiction.
\end{proof}

\begin{figure}
    \centering
    \includegraphics{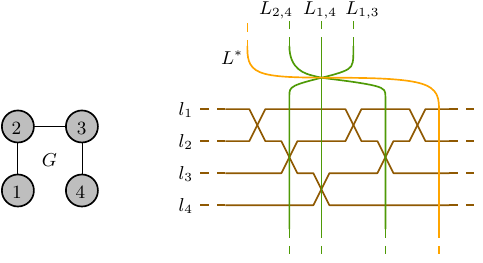}
    \caption{Reduction from approximating $\chi(G)$ to computing $\chi_{pl}(\mathcal{A})$.}
    \label{fig:np_reduction}
\end{figure}

\begin{proposition}
    Given an arrangement of pseudolines~$\mathcal{A}$, it is NP-hard to compute~$\chi_{pl}(\mathcal{A})$.
\end{proposition}
\begin{proof}
    In~\cite{zuckerman07} it was shown that for every $\varepsilon > 0$, given a graph $G$, approximating the chromatic number $\chi(G)$ within $n^{1-\varepsilon}$ is NP-hard. Via Lemma~\ref{lemma:reduction}, a polynomial time algorithm for computing $\chi_{pl}(\mathcal{A})$ would provide such an approximation of $\chi(G)$ in polynomial~time.
\end{proof}

\subsection{Avoiding monochromatic crossings of high degrees}

Even though $\chi_{pl}(\mathcal{A})$ and $\chi(G_o(\mathcal{A}))$ can differ by a multiplicative factor, as stated in Proposition~\ref{prop:gap}, for most arrangements, $\chi_{pl}(\mathcal{A})$ does not seem to be far away from $\chi(G_o(\mathcal{A}))$. This is why our focus lies now on a variant of pseudoline colorings: Instead of avoiding monochromatic crossings of any degree, including ordinary points, we only forbid crossings of certain degrees to be monochromatic.

\begin{lemma}\label{lemma:pseudoline_coloring_degree_interval}
    Let $l,r\in\mathbb{N}$ and let $\mathcal{A}$ be an arrangement of $n$ pseudolines. Then, using \[\left(\frac{4(l+r)}{l-1}n\right)^{\frac{1}{l-1}}\in\mathcal{O}(n^{\frac{1}{l-1}})\] colors, $\mathcal{A}$ can be colored avoiding monochromatic crossings of degree within $\{l, l+1, \cdots, l+r\}$.
\end{lemma}
\begin{proof}
    The case $l=2$ is trivial, so assume $l \geq 3$. For some integer $k$, choose for each pseudoline one out of $k$ colors randomly with uniform distribution. For each crossing $c$ of degree $\operatorname{deg}(c)\in\{l, l+1, \cdots, l+r\}$, consider the event $E(c)$ that $c$ is monochromatic. We have \[\mathbb{P}[E(c)]=\frac{1}{k^{\operatorname{deg}(c) - 1}}\leq \frac{1}{k^{l - 1}} =: p.\] Moreover, $E(c)$ is mutually independent of the set of events $E(c')$ for all crossings $c'$ in $\mathcal{A}$ that do not have a common pseudoline with $c$. These are all but at most \[\operatorname{deg(c)}\frac{n}{l-1}\leq \frac{l+r}{l-1}n =: d\] many. Applying Lovász Local Lemma (see~\cite{molloyReed02}) on the family of events $\{E(c)\}$, guarantees that with non-zero probability none of these events occur if $4pd\leq 1$, which is the case when \[k\geq \left(\frac{4(l+r)}{l-1}n\right)^{\frac{1}{l-1}}.\] The statement follows.
\end{proof}

Using Lemma~\ref{lemma:pseudoline_coloring_degree_interval} we are now able to prove Theorem~\ref{thm:pseudoline_coloring_degree_at_least_four}.

\begin{proof}[Proof of Theorem~\ref{thm:pseudoline_coloring_degree_at_least_four}]
    Assume that all crossings have degree at most $\sqrt{n}$. In this case, applying Lemma~\ref{lemma:pseudoline_coloring_degree_interval} with~$l=4$ and~$r=\sqrt{n}-4$ gives that \[\left(\frac{4\sqrt{n}}{3}n\right)^{\frac{1}{3}}\in\mathcal{O}(\sqrt{n})\] colors are sufficient in order to avoid monochromatic crossings of degree at least~$4$.

    Now, if~$\mathcal{A}$ contains crossings of degree higher than~$\sqrt{n}$, eliminate such a crossing by dropping the corresponding bundle of pseudolines from the arrangement. Repeat this step until all crossings have degree at most~$\sqrt{n}$. One requires less than~$\sqrt{n}$ iterations as each iteration removes more than~$\sqrt{n}$ pseudolines. By the argument above, the remaining arrangement~$\mathcal{A'}$ can be colored avoiding monochromatic crossings of degree at least $4$. When reinserting each of the at most~$\sqrt{n}$ removed bundles of pseudolines, one needs at most $2$ extra colors, in total~$2\sqrt{n}$ extra colors.
\end{proof}

If we only want to avoid monochromatic crossings of a single degree, then we can obtain a stronger result by applying again a result from the hypergraph coloring literature. The following theorem is due to Frieze and Mubayi \cite{FriezeMubayi13}. Its proof is also based on the probabilistic method, but in addition to Lovász Local Lemma it uses a rather involved analysis including the Chernoff bound. A hypergraph is \textit{$k$-uniform} if all its hyperedges have cardinality $k$.

\begin{theorem}[Frieze \& Mubayi, 2013]\label{thm:vertex_coloring_simple_hypergraphs}
    Fix $k\geq 3$. Let $\mathcal{H}$ be a $k$-uniform simple hypergraph with maximum degree $\triangle$. Then \[\chi(\mathcal{H}) \leq c\left( \frac{\triangle}{\log \triangle}\right)^{\frac{1}{k-1}},\] where $c$ only depends on $k$.
\end{theorem}

\begin{proposition}\label{prop:pseudoline_coloring}
    Fix some~$l\geq 3$ and let~$\mathcal{A}$ be an arrangement of~$n$ pseudolines. Then, the pseudolines in~$\mathcal{A}$ can be colored using~\[c\cdot\left(\frac{\operatorname{mx}(\mathcal{A})}{\log \operatorname{mx}(\mathcal{A})}\right)^{\frac{1}{l-1}}\in\mathcal{O}\left(\left(\frac{n}{\log n}\right)^{\frac{1}{l-1}}\right)\] colors avoiding monochromatic crossings of degree exactly
    $l$, where~$c$ only depends on~$l$.
\end{proposition}
\begin{proof}
    Recall that pseudoline coloring is equivalent to vertex coloring of $\mathcal{H}_{\text{line-vertex}}$. However, if we only aim for avoiding monochromatic crossings of degree $l$, then we only care about hyperedges of degree~$l$ and delete all other hyperedges. The resulting hypergraph is simple and $l$-uniform. Its maximum degree degree $\triangle$ equals~$\operatorname{mx}(\mathcal{A})$. The statement is hence a direct application of Theorem~\ref{thm:vertex_coloring_simple_hypergraphs}.
\end{proof}

\section{Conclusion and Future Work}

We consider Theorem~\ref{thm:crossing_coloring_face_respecting} as our main result. When coloring the crossings avoiding twice the same color along any pseudoline, Theorem~\ref{thm:crossing_coloring} is a direct application of the recently proven Erd{\H{o}}s-Faber-Lovász conjecture. However, for the specific hypergraphs induced by pseudoline arrangements, one could hope for a simple deterministic coloring procedure, like the one proposed in~\cite{changLawler88} that requires $\lceil(3/2)n-2\rceil$ colors.

We mentioned Conjecture~\ref{conj:crossing_coloring_dep_on_mx} as an open problem. One may also ask whether there always exists a coloring using~$n$ colors that satisfies the conditions of Theorem~\ref{thm:crossing_coloring_face_respecting} and Theorem~\ref{thm:crossing_coloring} simultaneously. Figure~\ref{fig:example_non_simultanously_colorable} shows an arrangement of $5$ pseudolines as a counterexample: coloring the crossings avoiding twice the same color around any cell and twice the same color along any peudoline requires at least $6$ colors. Are there arbitrarily large examples like that?

When it comes to pseudoline colorings, we asked whether~$\chi_{pl}(\mathcal{A})$ and~$\chi(G_o(\mathcal{A}))$ can differ by a factor larger than~$2$. Finally, in view of Lemma~\ref{lemma:pseudoline_coloring_degree_interval}, Theorem~\ref{thm:pseudoline_coloring_degree_at_least_four} and Proposition~\ref{prop:pseudoline_coloring}, we expect it to be possible to color the pseudolines of every arrangement using~$\mathcal{O}(n^{\frac{1}{l-1}})$ colors avoiding monochromatic crossings of degree at least~$l$.

\begin{figure}
    \centering
    \includegraphics{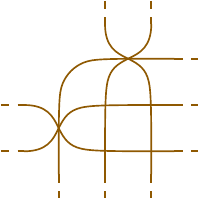}
    \caption{Arrangement of~$5$ pseudolines that cannot be simultaneously colored as in Theorem~\ref{thm:crossing_coloring_face_respecting} and Theorem~\ref{thm:crossing_coloring}.}
    \label{fig:example_non_simultanously_colorable}
\end{figure}

\bibliography{bibliography}

\end{document}